\documentclass[10pt]{amsart}
\usepackage{amssymb}
\usepackage{amsthm}
\usepackage{amsmath}
\usepackage{amsfonts}
\usepackage[mathscr]{eucal}
\usepackage{enumerate}

\newtheorem{thm}{Theorem}[section]
\newtheorem{cor}[thm]{Corollary}

\theoremstyle{definition}

\theoremstyle{remark}

\newtheorem*{ackn}{Acknowledgement}
\numberwithin{equation}{section}

\begin{document}

\begin{flushright}
 \small{\emph{To appear in Acta Scientiarum Mathematicarum (Szeged)}}
\end{flushright}

\title[Notes on the characterization of derivations]
{Notes on the characterization of derivations}

\author[E.~Gselmann]{Eszter Gselmann}
\address{
Department of Analysis
\\
Institute of Mathematics
\\
University of Debrecen
\\
P. O. Box: 12.
\\
Debrecen
\\
H--4010
\\
Hungary}

\email{gselmann@science.unideb.hu}


\begin{abstract}
Although the characterization of ring derivations has an extensive literature,
up to now, all of the characterizations have had the following
form: additivity and another property imply that the function in question is a derivation.
The aim of this note is to point out that derivations can be described via a single equation.
\end{abstract}

\thanks{This research has been supported by the Hungarian Scientific Research Fund (OTKA)
Grant NK 814 02 and by the T\'{A}MOP 4.2.1./B-09/1/KONV-2010-0007 project implemented
through the New Hungary Development Plan co-financed by the European Social Fund and
the European Regional Development Fund.}
\maketitle

\section{Introduction and preliminaries}

The purpose of this paper is to provide new characterization theorems on derivations.
At first, we will list some preliminary results that will be used in the sequel.
All of these statements and definitions can be found in Kuczma \cite{Kuc85} and
also in Zariski--Samuel \cite{ZS58}.

Let $Q$ be a commutative ring and let $P$ be a subring of $Q$.
A function $f:P\rightarrow Q$ is called a \emph{derivation} if it is additive,
i.e.,
\[
f(x+y)=f(x)+f(y)
\quad
\left(x, y\in P\right)
\]
and also satisfies the equation
\[
f(xy)=xf(y)+yf(x).
\quad
\left(x, y\in P\right)
\]

A fundamental example for derivations is the following.

Let $\mathbb{F}$ be a field, and let in the above definition $P=Q=\mathbb{F}[x]$
be the ring of polynomials with coefficients from $\mathbb{F}$. For a polynomial
$p\in\mathbb{F}[x]$, $p(x)=\sum_{k=0}^{n}a_{k}x^{k}$, define the function
$f:\mathbb{F}[x]\rightarrow\mathbb{F}[x]$ as
\[
f(p)=p',
\]
where $p'(x)=\sum_{k=1}^{n}ka_{k}x^{k-1}$ is the derivative of the polynomial $p$.
Then the function $f$ clearly fulfils
\[
f(p+q)=f(p)+f(q)
\]
and
\[
f(pq)=pf(q)+qf(p)
\]
for all $p, q\in\mathbb{F}[x]$. Hence $f$ is a derivation.

In $\mathbb{R}$ (the set of the real numbers) the identically zero function is
obviously a (trivial) derivation. It is difficult to find however another example,
because every real derivation has the following properties.
\begin{enumerate}[---]
  \item If $f:\mathbb{R}\rightarrow\mathbb{R}$ is a derivation, then $f(x)=0$
  for all $x\in\mathrm{algcl}\mathbb{Q}$ (the algebraic closure of the set of the rational numbers).
  \item If $f:\mathbb{R}\rightarrow\mathbb{R}$ is a derivation, and $f$ is measurable, or bounded above
  or below on the set which has positive Lebesgue measure, then $f$ is identically zero.
\end{enumerate}

In spite of the above properties, there exist non--trivial derivations in $\mathbb{R}$,
see Theorem 14.~2.~2. in \cite{Kuc85}.

The characterization of derivations has an extensive literature, the reader should consult for
instance Horinouchi--Kannappan \cite{HK71}, Jurkat \cite{Jur65},
Kurepa \cite{Kur64, Kur65} and also the two monographs Kuczma \cite{Kuc85} and
Zariski--Samuel \cite{ZS58}.

Nevertheless, to the best of the author's knowledge, all of the characterizations have the following
form: additivity and another property imply that the function in question is a derivation.
We intend to show that derivations can be characterized by one single functional equation.

More precisely, we would like to examine whether the equations occurring in the definition of
derivations are independent in the following sense.
Let $\lambda, \mu\in Q\setminus\left\{0\right\}$  be arbitrary,
$f:P\rightarrow Q$ be a function and consider the
equation
\[
\lambda\left[f(x+y)-f(x)-f(y)\right]+
\mu\left[f(xy)-xf(y)-yf(x)\right]=0.
\quad
\left(x, y\in P\right)
\]
Clearly, if the function $f$ is a derivation, then this equation holds.
In the next section we will investigate the opposite direction, and it will be proved that
under some assumptions on the rings $P$ and $Q$, derivations can be
characterized through the above equation.
This result will be proved as a consequence of the main theorem that will be devoted
to the equation
\[
f(x+y)-f(x)-f(y)=g(xy)-xg(y)-yg(x),
\quad
\left(x, y\in P\right)
\]
where $f, g:P\rightarrow Q$ are unknown functions. 

We remark that similar investigations were made by
Dhombres \cite{Dho88}, Ger \cite{Ger98, Ger00} and also by Ger--Reich \cite{GR10} concerning ring homomorphisms. 
For instance, in Ger \cite{Ger98} the following theorem was proved. 
\begin{thm}
 Let $X$ and $Y$ are two rings, and assume that 
for all $x\in X$ there exists $e_{x}\in X$ such that $x e_{x}=x$, suppose further 
that $Y$ has no elements of order 2 and does not admit zero divisors. 
If $f$ is a solution of the equation 
\[
 f(x+y)+f(xy)=f(x)+f(y)+f(x)f(y) 
\qquad 
\left(x, y\in X\right)
\]
such that $f(0)=0$, 
then either $3f$ is even and $3f(2x)=0$ for all $x\in X$, 
or $f$ yields a homomorphism between $X$ and $Y$. 
\end{thm}

During the proof of the main result the celebrated \emph{cocycle equation}
will play a key role.
About this equation one can read e.g. in
Acz\'{el} \cite{Acz65}, Davison--Ebanks \cite{DE95},
Ebanks \cite{Eba82}, Erd\H{o}s \cite{Erd59},  
Hossz\'{u} \cite{Hos63} and also in Jessen--Karpf--Thorup \cite{JKT69}. 
In the next section we will however utilize only
Theorem 3 of Ebanks \cite{Eba79}, which reads as follows.

\begin{thm}\label{T1.1}
Let $A$ be an integral domain and $X$ a uniquely $A$--divisible unitary module over $A$. 
Then, $F, G:A^{2}\to X$ satisfy equations 
\begin{equation}
 \tag{$\alpha$} F(a, b)=F(b, a) \quad (a, b\in A), 
\end{equation}

\begin{equation}
\tag{$\beta$} F(a+b, c)+F(a, b)=F(a, b+c)+F(b, c) \quad (a, b, c\in A), 
\end{equation}
\begin{equation}
\tag{$\gamma$} G(a, b)=G(b, a) \quad (a, b\in A), 
\end{equation}
\begin{equation}
\tag{$\delta$} cG(a, b)+G(ab, c)=aG(b, c)+G(a, bc) \quad (a, b, c\in A), 
\end{equation}
\begin{equation}
\tag{$\varepsilon$} F(ac, bc)-cF(a, b)=G(a+b, c)-G(a, c)-G(b, c) \quad (a, b, c\in A), 
\end{equation}
\begin{equation}
\tag{$\zeta$} \sum_{i=1}^{p}F(1, i1)=0, \quad p=\mathrm{char}A, 
\end{equation}
if and only if there is a map $f:A\to X$ representing $F$ and $G$ 
through equations 
\begin{equation}
\tag{A} F(a, b)=f(a+b)-f(a)-f(b),
\end{equation} 
respectively, 
\begin{equation}
\tag{B} G(a, b)=f(ab)-af(b)-bf(a). 
\end{equation}
Moreover, if $A$ is ordered, the same result holds with $A$ replaced by 
\[A_{+}=\left\{a\in A\, \vert \, a>0\right\}.\] 
\end{thm}

From this theorem with the choice $F\equiv 0$ (and interchanging the roles of $f$ and $g$) 
the following statement can be
obtained immediately.

\begin{thm}\label{T1.2}
Let $A$ be a commutative ring, $X$ be a module over $A$ and
$f:A\rightarrow X$ be a function such that
\[
f(ab)=af(b)+bf(a)
\qquad
\left(a, b\in A\right). 
\]
Then the function $F:A\times A\rightarrow X$ defined by
$(A)$ fulfils
\begin{equation}
\tag{$\alpha$} F(a, b)=F(b, a), 
\end{equation}
equation $(\beta)$ is satisfied and also
\begin{equation}
\tag{$\eta$} F(ac, bc)=cF(a, b)
\end{equation}
holds for any $a, b, c\in A$.

Furthermore, in case $A$ is an integral domain and $X$ is a
unitary module over $A$ which is uniquely $A$--divisible, then
the function $F$ defined by the function $f$ through equation $(A)$ is the only function
which satisfies equations $(\alpha)$, $(\beta)$ and $(\eta)$.
\end{thm}

\section{The main result}

Our main result is contained in the following

\begin{thm}\label{T2.1}
Let $\mathbb{F}$ be a field, $X$ be a vector space over $\mathbb{F}$ and
$f, g:\mathbb{F}\rightarrow X$ be functions such that
\begin{equation}
\tag{$\mathscr{E}$}
f(x+y)-f(x)-f(y)=g(xy)-xg(y)-yg(x)
\end{equation}
holds for all $x, y\in\mathbb{F}$.
Then, and only then, there exist additive functions $\alpha, \beta\colon\mathbb{F}\to X$  and a 
function 
$\varphi\colon \mathbb{F}\to X$ with the property
\[
 \varphi(xy)=x\varphi(y)+y\varphi(x) 
\qquad 
\left(x, y\in\mathbb{F}\right), 
\]
 such that 
\[
 f(x)=\beta(x)+\frac{1}{2}\alpha(x^{2})-x\alpha(x)
\qquad 
\left(x, y\in\mathbb{F}\right)
\]
and 
\[
 g(x)=\varphi(x)+\alpha(x) 
\qquad 
\left(x, y\in\mathbb{F}\right)
\]
are satisfied. 
\end{thm}

\begin{proof}
Define the functions $\mathscr{C}_{f}, \mathscr{D}_{f}, \mathscr{C}_{g}$ and $\mathscr{D}_{g}$
on $\mathbb{F}\times \mathbb{F}$ by
\[
\begin{array}{lcl}
\mathscr{C}_{f}(x, y)&=&f(x+y)-f(x)-f(y) \\
\mathscr{D}_{f}(x, y)&=&f(xy)-xf(y)-yf(x) \\
\mathscr{C}_{g}(x, y)&=&g(x+y)-g(x)-g(y) \\
\mathscr{D}_{g}(x, y)&=&g(xy)-xg(y)-yg(x),
\end{array}
\]
respectively.
In view of Theorem \ref{T1.1}., we immediately get that the pairs
$(\mathscr{C}_{f}, \mathscr{D}_{f})$ and $(\mathscr{C}_{g}, \mathscr{D}_{g})$
fulfill the system of equations $(\alpha)$--$(\varepsilon)$.
Furthermore, equation $(\mathscr{E})$ yields that
\begin{equation}
\tag{$\mathscr{E}^{\ast}$}\mathscr{C}_{f}(x, y)=\mathscr{D}_{g}(x, y)
\end{equation}
for all $x, y\in \mathbb{F}$.
Due to equation $(\varepsilon)$,
\begin{equation}\label{Eq2.1}
\mathscr{C}_{g}(xz, yz)-z\mathscr{C}_{g}(x, y)=\mathscr{D}_{g}(x+y, z)-\mathscr{D}_{g}(x, z)-\mathscr{D}_{g}(y, z)
\end{equation}
holds for all $x, y, z\in\mathbb{F}$.
Interchanging the role of $x$ and $z$ in the previous equation, we obtain that
\begin{equation}\label{Eq2.2}
\mathscr{C}_{g}(xz, xy)-x\mathscr{C}_{g}(z, y)=\mathscr{D}_{g}(y+z, x)-\mathscr{D}_{g}(z, x)-\mathscr{D}_{g}(y, x)
\end{equation}
for any $x, y, z\in\mathbb{F}$.
Let us subtract equation \eqref{Eq2.2} from \eqref{Eq2.1}, to obtain
\begin{multline*}
\mathscr{C}_{g}(xz, yz)-z\mathscr{C}_{g}(x, y)
-\mathscr{C}_{g}(xz, xy)+x\mathscr{C}_{g}(z, y)
\\
=
\mathscr{D}_{g}(x+y, z)-\mathscr{D}_{g}(x, z)-\mathscr{D}_{g}(y, z)
-\mathscr{D}_{g}(z+y, x)+\mathscr{D}_{g}(z, x)+\mathscr{D}_{g}(y, x).
\end{multline*}
Because of $(\mathscr{E}^{\ast})$, the function $\mathscr{D}_{g}$ can be replaced by
$\mathscr{C}_{f}$. This implies however that
\begin{multline*}
\mathscr{C}_{g}(xz, yz)-z\mathscr{C}_{g}(x, y)
-\mathscr{C}_{g}(xz, xy)+x\mathscr{C}_{g}(z, y)
\\
=
\mathscr{C}_{f}(x+y, z)+\mathscr{C}_{f}(x, y)
-\mathscr{C}_{f}(x, y+z)-\mathscr{C}_{f}(y, z)=0,
\end{multline*}
for all $x, y, z\in\mathbb{F}$, where we used that the function
$\mathscr{C}_{f}$ fulfils $(\alpha)$ and $(\beta)$.
This equation with $z=1$ yields that
\[
\mathscr{C}_{g}(x, xy)=x\mathscr{C}_{g}(1, y),
\]
or if we replace $y$ by $\dfrac{y}{x}$, $(x\neq 0)$
\[
\mathscr{C}_{g}(x, y)=x\mathscr{C}_{g}\left(1, \frac{y}{x}\right).
\quad \left(x, y\in\mathbb{F}, x\neq 0\right)
\]
We will show that from this identity the homogeneity of
$\mathscr{C}_{g}$ follows. Indeed, let $t, x, y\in\mathbb{F},
t, x\neq 0$ be arbitrary, then
\[
\mathscr{C}_{g}(tx, ty)=tx\mathscr{C}_{g}\left(1, \frac{ty}{tx}\right)=
tx\mathscr{C}_{g}\left(1, \frac{y}{x}\right)=t\mathscr{C}_{g}(x, y).
\]
If $x=0$, we get from equation $\left(\mathscr{E}^{\ast}\right)$ that $\mathscr{C}_{g}(0,
0)=0$, thus for arbitrary $t\in\mathbb{F}$,
\[
\mathscr{C}_{g}(t 0, t 0)= 0=t\mathscr{C}_{g}(0, 0).
\]
Furthermore, in case $t=0$, then for any $x, y\in \mathbb{F}$
\[
\mathscr{C}_{g}(tx, ty)=\mathscr{C}_{g}(0, 0)=0=t\mathscr{C}_{g}(x, y).
\]
This means that the function $\mathscr{C}_{g}$ is homogeneous and
fulfils equations $(\alpha)$ and $(\beta)$. In view of Theorem
\ref{T1.2}., there exists a function $\varphi\colon\mathbb{F}\to X$ such that 
\[
 \varphi(xy)=x\varphi(y)+y\varphi(x)
\qquad 
\left(x, y\in\mathbb{F}\right)
\]
and 
\[
 \mathscr{C}_{g}(x, y)=\varphi(xy)-x\varphi(y)-y\varphi(x)
\qquad 
\left(x, y\in\mathbb{F}\right)
\]
hold. 
Due to the definition of the function $\mathscr{C}_{g}$, this yields that 
\[
 g(x)=\varphi(x)+\alpha(x) 
\qquad 
\left(x\in\mathbb{F}\right), 
\]
where the function $\varphi$ fulfils the above identity and $\alpha\colon\mathbb{F}\to X$ is additive. 
Writing this representation of the function $g$ into equation ($\mathscr{E}$), we have that 
\begin{equation}\label{Eq2.3}
 f(x+y)-f(x)-f(y)=\alpha(xy)-x\alpha(y)-y\alpha(x) 
\qquad 
\left(x, y\in\mathbb{F}\right). 
\end{equation}
Since the function $\alpha$ is additive, the twp place function 
\[
 \mathscr{D}_{\alpha}(x, y)=\alpha(xy)-x\alpha(y)-y\alpha(x) 
\qquad 
\left(x, y\in\mathbb{F}\right)
\]
is a symmetric, biadditive function. 
Therefore, $\mathscr{D}_{\alpha}$ can be written as the Cauchy difference of its trace, that is 
\begin{multline*}
 \mathscr{D}_{\alpha}(x, y)=
\frac{1}{2}\alpha((x+y)^{2})-(x+y)\alpha(x+y)
\\
-\left(\frac{1}{2}\alpha(x^{2})-x\alpha(x)\right)
-\left(\frac{1}{2}\alpha(y^{2})-y\alpha(y)\right)
\qquad 
\left(x, y\in\mathbb{F}\right). 
\end{multline*}
In view of equation \eqref{Eq2.3}, this yields that the function 
\[
 x\mapsto f(x)-\left(\frac{1}{2}\alpha(x^{2})-x\alpha(x)\right) 
\qquad 
\left(x\in\mathbb{F}\right)
\]
is additive. 
Thus there exists an additive function $\beta\colon\mathbb{F}\to X$ such that 
\[
 f(x)=\beta(x)+\frac{1}{2}\alpha(x^{2})-x\alpha(x)
\qquad 
\left(x\in\mathbb{F}\right). 
\]

\end{proof}

According to Theorem \ref{T1.1}, with the aid of the previous
result, the following corollary can immediately be obtained. 

\begin{cor}
Let $\mathbb{F}$ be an ordered field, $X$ be a vector space over $\mathbb{F}$,
$\mathbb{F}_{+}=\left\{x\in\mathbb{F}\vert x>0\right\}$ and
$f, g:\mathbb{F}_{+}\rightarrow X$ be functions such that
\[
f(x+y)-f(x)-f(y)=g(xy)-xg(y)-yg(x)
\]
holds for all $x, y\in\mathbb{F}_{+}$.
Then the functions $f$ and $g$ can be extended to functions
$\widetilde{f}, \widetilde{g}:\mathbb{F}\rightarrow X$ such that
\[
\widetilde{f}(x)=\beta(x)+\frac{1}{2}\alpha(x^{2})-x\alpha(x)
\qquad
\left(x, y\in\mathbb{F}\right)
\]
and
\[
\widetilde{g}(x)=\varphi(x)+\alpha(x), 
\qquad
\left(x, y\in\mathbb{F}\right)
\]
where $\alpha, \beta\colon\mathbb{F}\to X$ are additive function and $\varphi\colon\mathbb{F}\to X$ fulfils 
\[
 \varphi(xy)=x\varphi(y)+y\varphi(x) 
\qquad 
\left(x, y\in\mathbb{F}\right). 
\]
\end{cor}

From Theorem \ref{T2.1}. with $g(x)=-\dfrac{\mu}{\lambda}f(x)$ the following corollary can be
derived easily.

\begin{cor}\label{C2.2}
Let $\mathbb{F}$ be a field and $X$ be a vector space over $\mathbb{F}$,
$\lambda, \mu\in \mathbb{F}\setminus\left\{0\right\}$
be arbitrarily fixed.
Then the function $f:\mathbb{F}\rightarrow X$ is a derivation if and only if
\[
\lambda\left[f(x+y)-f(x)-f(y)\right]+
\mu\left[f(xy)-xf(y)-yf(x)\right]=0
\]
holds for all $x, y\in \mathbb{F}$.
\end{cor}

\begin{ackn}
I wish to express my gratitude to the anonymous referee 
for several helpful comments and for drawing my attention to the paper \cite{Eba79}. 
\end{ackn}

\end{document}